\documentclass[12pt,a4paper]{article}
\usepackage{amsfonts,amsmath,amssymb}
\usepackage{oldlfont}
\usepackage[cp1250]{inputenc}
\usepackage[T1]{fontenc}
\usepackage[width=17cm,height=25cm]{geometry}

\usepackage{color}

\usepackage{pgf,tikz}
\usetikzlibrary{arrows,patterns}
\usetikzlibrary{matrix,positioning}

\newtheorem{theorem}{\bf Theorem}[section]
\newtheorem{corollary}[theorem]{Corollary}
\newtheorem{lemma}[theorem]{Lemma}
\newtheorem{proposition}[theorem]{Proposition}

\newtheorem{remark}[theorem]{Remark}

\newtheorem{question}[theorem]{Question}

\newcommand{\proof}{\noindent{\bf Proof.\ }}
\newcommand{\qed}{\hfill $\square$ \bigskip}

\begin{document}

\title{Toll number of the Cartesian and the lexicographic product of graphs}

\author{Tanja Gologranc $^{a,b}$ \and
Polona Repolusk $^{a}$}

\maketitle

\begin{center}
$^a$ Faculty of Natural Sciences and Mathematics, University of Maribor, Slovenia\\
{\tt tanja.gologranc1@um.si\\
polona.repolusk@um.si}\\
\medskip

$^b$ Institute of Mathematics, Physics and Mechanics, Ljubljana, Slovenia\\
\end{center}

\begin{abstract}

Toll convexity is a variation of the so-called interval convexity.  A tolled walk $T$ between $u$ and $v$ in $G$ is a walk of the form $T: u,w_1,\ldots,w_k,v,$ where $k\ge 1$, in which $w_1$ is the only neighbor of $u$ in $T$ and $w_k$ is the only neighbor of $v$ in $T$. As in geodesic or monophonic convexity, toll interval between $u,v\in V(G)$ is a set $T_G(u,v)=\{x\in V(G)\,:\,x \textrm{ lies on a tolled walk between } u \textrm{\, and\,} v\}$. A set of vertices $S$ is toll convex, if $T_{G}(u,v)\subseteq S$ for all $u,v\in S$.
First part of the paper reinvestigates the characterization of convex sets in the Cartesian product of graphs. Toll number and toll hull number of the Cartesian product of two arbitrary graphs is proven to be 2. The second part deals with the lexicographic product of graphs. It is shown that if $H$ is not isomorphic to a complete graph, $tn(G \circ H) \leq 3\cdot tn(G)$. We give some necessary and sufficient conditions for $tn(G \circ H) = 3\cdot tn(G)$. Moreover, if
$G$ has at least two extreme vertices, a complete characterization is given.
Also graphs with  $tn(G \circ H)=2$ are characterized - this is the case iff $G$ has an universal vertex and $tn(H)=2$.
Finally, the formula for $tn(G \circ H)$ is given - it is described in terms of the so-called toll-dominating triples.



\end{abstract}

\newpage

\section{Introduction}\label{Int}

Theory of convex structures developed from the classical convexity in Euclidean spaces and resulted in the {\it abstract convexity} theory. It is based on three natural conditions, imposed on a family of subsets of a given set. All three axioms hold in the so-called {\it interval convexity} which was emphasized in~\cite{vel-93} as one of the most natural ways for introducing convexity. An interval $I:X \times X \rightarrow 2^X$ has the property that $x,y\in I(x,y)$, and convex sets are defined as the sets $S$ in which all intervals between elements from $S$ lie in $S$. Several interval structures have been introduced in graphs and the interval function $I$ is usually defined by a set of paths of a certain type between two vertices. In this way, shortest paths yield geodesic intervals, induced paths yield monophonic intervals and each type of interval give rise to the corresponding convexity, see~\cite{chmusi-05,Pela} for some basic types of intervals/convexities.

There are many properties that were investigated in different interval convexities. One of the most natural arises from the abstract convexity theory, i.e.\ convex geometry property. The problem is whether a given convexity presents a convex geometry (i.e.\ enjoys the Minkowski--Krein--Milman property) which is related to rebuilding convex sets from extremal elements. In the case of monophonic convexity exactly chordal graphs are convex geometries, while in the geodesic convexity these are precisely Ptolemaic graphs (i.e.\ distance-hereditary chordal graphs), see~\cite{fj-86}.
Graph convexity, for which exactly the interval graphs are convex geometry, was investigated and introduced in~\cite{abg}. As interval graphs were investigated, authors used the concept from~\cite{al-14}, where interval graphs were characterized in terms of tolled paths. In~\cite{abg} the authors defined a toll convexity and proved that in this convexity the interval graphs are precisely the graphs which are convex geometry. Toll convexity arises from tolled walks, which are a generalization of monophonic paths, as any monophonic path is also a tolled walk. The paper also consider other properties of toll convexity that were already investigated in terms of other types of convexities. They are focused on two standard invariants with respect to toll convexity, the toll number and the t-hull number of a graph, that arise from similar invariants in terms of geodesic convexity, i.e.\ the geodetic number and the hull number of a graph.  

The geodetic and the hull number of a graph are two graph theoretic parameters introduced about 30 years ago~\cite{Evertt85,hlt} and intensively studied after that, see~\cite{CHZ02,cpz,Pela} for more results on this topic. Both invariants were also studied in graph products~\cite{bkh, bkt,chm, CC} and in terms of other types of convexities~\cite{cz,stg}. In~\cite{abg} both invariants were introduced in terms of toll convexity and it was proved that if $G$ is an interval graph, then toll number and t-hull number coincide with the power of the set of extreme vretices of $G$. They also studied toll number and t-hull number of trees. In this paper we will focus in these two invariants on the Cartesian and the lexicographic product of two graphs.  

In the following section we present main definitions and results form~\cite{abg} that will be needed all over the paper. In Section~\ref{s:Car} we focus in the toll number and t-hull number of the Cartesian product of graphs. We first present a counterexample to the characterization of t-convex sets in the Cartesian product from~\cite{abg} in which the authors missed one condition. Then we fix the characterization and use the result to prove that the $t$-hull number of the Cartesian product of two arbitrary non-complete graphs equals 2. Then we study toll number of the Cartesian product of two arbitrary graphs and deduce that it also equals 2. In Section~\ref{s:Lex} we again use the result from~\cite{abg} and prove that t-hull number of the lexicographic product of two connected non-trivial graphs $G$ and $H$, where $H$ is not complete, equals 2. Then we prove some bounds for the toll number of the lexicographic product of two graphs and give some necessary and sufficient conditions for  $tn(G \circ H)=3 \cdot tn(G)$. 
If $G$ has two extreme vertices (i.e. $|Ext(G)|=2$), we give a complete characterization of graphs with $tn(G \circ H) = 3\cdot tn(G)$. 
We also characterize graphs with $tn(G \circ H)=2$ and finally we establish a formula that expresses the exact toll number of $G \circ H$ using the new concept, obtained from the same idea as geodominating triple in~\cite{bkt}.


\section{Preliminaries}\label{Pre}

Graphs in this paper will be undirected, without loops or multiple edges.
Let $G$ be a graph. The {\it distance} $d_{G}(u,v)$ between vertices $u,v\in V(G)$ is the length of a shortest path between $u$ and $v$ in $G.$  
The {\it{geodesic interval}} $I_{G}(u,v)$ between vertices $u$ and $v$ is the set of all vertices that lie on some shortest path between $u$ and $v$ in $G$, i.e. $I_{G}(u,v)=\{x\in V(G)\,:\,d_G(u,x)+d_G(x,v)=d_G(u,v)\}$.
A subset $S$ of $V(G)$ is {\it{geodesically convex}} (or {\it g-convex}) if $I_{G}(u,v)\subseteq S$ for all $u,v\in S$. Let $S$ be a set of vertices of a graph $G$. Then the geodetic closure $I_G \left[ S \right]$ is the union of geodesic intervals between all pairs of vertices from $S$, that is, $I_G \left[ S \right] = \bigcup_{u,v \in S}I_G(u,v)$. A set $S$ of vertices of $G$ is a {\it geodetic set} in $G$ if $I_G \left[ S \right] = V (G)$. The size of a minimum geodetic set in a graph $G$ is called the {\it geodetic number} of $G$ and denoted by $g(G)$. Given a subset $S \subseteq V(G)$, the {\it convex hull} $\left[ S \right]$ of $S$ is the smallest convex set that contains $S$. We say that $S$ is a {\it hull set} of $G$ if $\left[ S \right]=V(G)$. The size of a minimum hull set of $G$ is the {\it hull number} of $G$, denoted by $hn(G)$. Indices above may be omitted, whenever the graph $G$ is clear from the context. 

There are also many other graph convexities such as monophonic convexity, all-path convexity, Steiner convexity and so on, which are also interesting in terms of the smallest sets whose closure is the whole vertex set and in terms of hull numbers (with respect to the chosen convexity). For more detailes on this topic see surveys~\cite{bkt-10,cpz}, the book~\cite{Pela} and the paper~\cite{Steiner-spa}. In this paper, these two invariants will be investigated in terms of the so called toll convexity.

Let $u$ and $v$ be two different non-adjacent vertices in $G$. A {\it tolled walk} $T$ between $u$ and $v$ in $G$ is a sequence of vertices of the form $$T: u,w_1,\ldots,w_k,v,$$ where $k\ge 1$, which enjoys the following three conditions:

\begin{itemize}
\item $w_iw_{i+1}\in E(G)$ for all $i$,
\item $uw_i\in E(G)$ if and only if $i=1$,
\item $vw_i\in E(G)$ if and only if $i=k$.
\end{itemize}

In other words, a tolled walk is any walk between $u$ and $v$ such that $u$ is adjacent only to the second vertex of the walk and $v$ is adjacent only to the second-to-last vertex of the walk. For $uv\in E(G)$ we let $T:u,v$ be a tolled walk as well and the only tolled walk that starts and ends in the same vertex $v$ is $v$ itself. 
We define $T_G(u,v)=\{x\in V(G)\,:\,x \textrm{ lies on a tolled walk between } u \textrm{\, and\,} v\}$ to be the {\it toll interval} between $u$ and $v$ in $G$. Finally, a subset $S$ of $V(G)$ is {\it{toll convex}} (or {\it t-convex}) if $T_{G}(u,v)\subseteq S$ for all $u,v\in S$. The {\it toll closure} $T_G[S]$ of a subset $S\subseteq V(G)$ is the union of toll intervals between all pairs of vertices from $S$, i.e. $T_G[S]=\cup_{u,v\in S} T_G(u,v)$. If $T_G[S]=V(G)$, we call $S$ a {\it toll set} of a graph $G$. The order of a minimum toll set in $G$ is called the {\it toll number} of $G$ and is denoted by $tn(G)$. Again, when graph is clear from the context, indices may be omitted.

A {\it t-convex hull} of a set $S\subseteq V(G)$ is defined as the intersection of all t-convex sets that contain $S$ and is denoted by $[S]_t$. A set $S$ is a {\it t-hull set} of $G$ if its t-convex hull $[S]_t$ coincides with $V(G)$. The {\it t-hull number} of $G$ is the size of a minimum t-hull set and is denoted by $th(G)$.
Given the toll interval $T\,:\,V \times V \rightarrow 2^V$ and a set $S\subset V(G)$ we define $T^k(S)$ as follows: $T^0(S)=S$ and $T^{k+1}(S)=T(T^k(S))$ for any $k\geq 1$. Note that $[S]_t= \bigcup_{k \in \mathbb{N}} T^k(S)$. From definitions we immediately infer that every toll set is a t-hull set, and hence $th(G)\leq tn(G)$. Since every geodetic set is a toll set, we have $tn(G)\leq g(G)$.

A vertex $s$ from a convex set $S$ is an {\em extreme vertex} of $S$, if $S-\{s\}$ is also convex. Thus also extreme vretices can be defined in terms of different graph convexities. In terms of geodesic and monophonic convexity, the extreme vertices are exactly simplicial vertices, i.e.\ vertice whose closed neighborhoods induce complete graphs. For toll convexity, any extreme vertex is also a simplicial but the converse is not necessary true, see~\cite{abg}. The set of all extreme vertices of a graph $G$ will be denoted by $Ext(G)$. The set of extreme vertices (with respect to toll convexity) is contained in any toll set of $G$ and even more, it is contained in any t-hull set of $G$, i.e.\ $|Ext(G)| \leq th(G) \leq tn(G)$. The assertion holds also in other convexities~\cite{CHZ02,Evertt85}. Graph $G$ with $|Ext(G)|=tn(G)$ is called {\it{extreme complete.}}

Recall that for all of the standard graph products, the vertex set of the product of graphs $G$ and $H$ is equal to $V(G)\times V(H)$. In the {\it{lexicographic product}} $G\circ H$ (also denoted by $G[H]$), vertices $(g_{1},h_{1})$ and $(g_{2},h_{2})$ are adjacent if either $g_{1}g_{2}\in E(G)$ or ($g_{1}=g_{2}$ and $h_{1}h_{2}\in E(H)$). In the
\emph{Cartesian product} $G\Box H$ of graphs $G$ and $H$ two vertices $(g_{1},h_{1})$ and $(g_{2},h_{2})$ are adjacent when ($g_1g_2\in E(G)$ and $h_1=h_2$) or ($g_1=g_2$ and $h_1h_2\in E(H)$).

Let $G$ and $H$ be graphs and $*$  be one of the two graph products under consideration. For a vertex $h\in V(H)$, we call the set $G^{h}=\{(g,h)\in V(G * H):g\in
V(G)\}$ a $G$-\emph{layer} of $G * H$. By abuse of notation we will also consider $G^{h}$ as the corresponding induced subgraph. Clearly $G^{h}$ is isomorphic to $G$. For $g\in V(G)$, the $H$-\emph{layer} $^g\!H$ is defined as $^g\!H =\{(g,h)\in V(G * H)\,:\,h\in V(H)\}$. We will again also consider $^g\!H$ as an induced subgraph and note that it is isomorphic to $H$. A map $p_{G}:V(G * H)\rightarrow V(G)$, $p_{G}(g,h) = g$ is the \emph{projection} onto $G$ and $p_{H}:V(G * H)\rightarrow V(H)$, $p_{H}(g,h) = h$ the \emph{projection} onto $H$. We say that $G * H$ is \emph{non-trivial} if both factors are graphs on at least two vertices. For more details on graph
products see~\cite{ImKl}. 

Let $G$ be a connected graph. A vertex $v\in V(G)$ is a {\it cut vertex}, if $G-\lbrace v \rbrace$ is not connected. A set $S \subseteq V(G)$ is a {\it separating set} if $G-S$ is not connected. 

\vspace{2mm}

Finally, we mention two useful results proved in~\cite{abg}, a characterization of a vertex $v$ from a tolled walk and a characterization of t-convex set in a graph $G$ using separating sets.

\begin{lemma}\cite[Lemma 2.3]{abg} 
\label{l:abg} A vertex $v$ is in some tolled walk between two non-adjacent vertices $x$ and $y$ if and only if $N[x] - \{v\}$ does not separate $v$ from $y$ and $N[y] - \{v\}$ does not separate $v$ from $x$.
\end{lemma}

\begin{proposition}\cite[Proposition 2.4]{abg}
\label{p:separate} Let $G$ be a graph. A subset $C$ of $V(G)$ is t-convex if and only if for every $x,y\in C$ and every $v\in V(G)-C$ the set $N[x]-\{v\}$ separates $v$ from $y$ or the set $N[y]-\{v\}$ separates $v$ from $x$.
\end{proposition}

\section{The Cartesian product}\label{s:Car}

In this section we prove that the t-hull number of the Cartesian product of two arbitrary non-trivial graphs equals 2. We give two different proofs for this result: it follows from the characterization of t-convex sets in the Cartesian product and also from the computation of the toll number. The main result of this section says that the toll number of the Cartesian product of two arbitrary non-trivial graphs equals 2.

We start with the following theorem from~\cite{abg}.

\begin{theorem}\cite[Theorem 5.3]{abg}\label{th:car} Let $G\Box H$ be a non-trivial, connected Cartesian product. A proper subset $Y$ of $V(G\Box H)$ which does not induce a complete graph is t-convex if and only if $Y=V(G_1)\times V(H_1)$ where one factor, say $H_1$, equals $H$, which is a complete graph, and $G_1$ is isomorphic to $P_k$ where every inner vertex of the path has degree 2 in $G$.
\end{theorem}

This theorem could be useful in solving the problem of t-hull number of the Cartesian product graphs, as the only possibility for the existence of a non-trivial, proper t-convex set in the product is that one factor is a complete graph.
The problem is, that the authors in~\cite{abg} missed one additional condition in the characterization of the t-convex sets in the Cartesian product graph, i.e.\ the necessary condition for $Y \subseteq V(G \Box H)$ being t-convex is too weak. Figure \ref{protiprimer} shows a counterexample to the Theorem~\ref{th:car}. In a graph $C_5 \Box K_3$, a set of black vertices 
is obviously not t-convex, eventhough it fulfills desired properties of the theorem.

\begin{figure}[h]
\centering
\begin{tikzpicture}[scale=1]
\foreach \x  in {0,1,2}
	{
	\draw (0,\x) -- (4,\x);
	\draw (0,\x) to [bend right=20] (4,\x);
	}
\foreach \x  in {0,1,2,3,4}
	{
	\draw (\x,0) -- (\x,2);
	\draw (\x,0) to [bend left=20] (\x,2);
	}

\foreach \x  in {0,4}
	\foreach \y  in {0,1,2}
		{
		\filldraw [fill=white, draw=black,thick] (\x,\y) circle (3pt);
		}
\foreach \x  in {1,2,3}
	\foreach \y  in {0,1,2}
		{
		\filldraw [fill=black, draw=black,thick] (\x,\y) circle (3pt);
		}
\end{tikzpicture}
\caption{A counterexample for the Theorem 5.3 in \cite{abg}.\label{protiprimer}}
\end{figure}
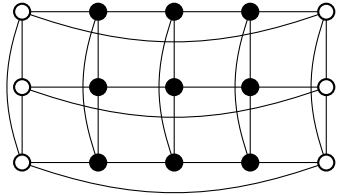

The proof from~\cite{abg} that the condition of the characterization is necessary is correct, thus we will just tighten the condition of the characterization and use this correct part from~\cite{abg} in our proof, while the direction that the condition is sufficient will be completely proved in this paper.


Therefore we give a new characterization of t-convex sets in the Cartesian product of graphs.

\begin{theorem}
\label{th:convex} Let $G\Box H$ be a non-trivial, connected Cartesian product. A proper subset $Y$ of $V(G\Box H)$ which does not induce a complete graph is t-convex if and only if $Y=V(G_1)\times V(H_1)$ where one factor, say $H_1$, equals $H$, which is a complete graph, and $G_1$ induces a path $P=v_1, \ldots , v_k$ in $G$ where every inner vertex of $P$ has degree 2 in $G$ and $P$ is the only $v_1,v_k$-path in $G$.
\end{theorem}

\proof 
Let $Y$ be a proper t-convex subset of $V(G \Box H)$, which does not induce a complete graph. From~\cite{abg} it follows that $Y=V(G_1)\times V(H_1)$ where one factor, say $H_1$, equals $H$, which is a complete graph, and $G_1$ induces a path $P=v_1, \ldots , v_k$ in $G$ where every inner vertex of the path has degree 2 in $G$. 

First observe that the length of any $v_1,v_k$-path $R$ in $G$ different from $P$ is at least 2. If $P$ contains just two vertices $v_1$ and $v_2$ then $v_1$ and $v_2$ are adjacent and any $v_1,v_2$ path different from $P$ contains at least three vertices. Otherwise, if $P$ contains more than two vertices then $v_1$ and $v_k$ are not adjacent in $G$, as $P$ induces a path in $G$. Thus the length of $R$ is at least two.

Suppose that there exists a $v_1,v_k$-path in $G$ different from $P$. Let $R=v_1,u_2,\ldots ,$ $u_l,v_k$ be such path with the shortest length and let $h$ and $h'$ be arbitrary different vertices from $H$. As $H$ is a complete graph $h$ and $h'$ are adjacent in $H$. Since $R \neq P$ and all inner vertices of $P$ have degree 2 in $G$, $R \cap P =\{v_1,v_k\}.$ Thus $(u_i,h) \notin Y$ for any $i \in \{1,\ldots , l\}$ and $(v_1,h),(u_2,h), (u_2,h'), (u_3,h'), \ldots ,$ $ (u_l,h'),(v_k,h')$ is a tolled walk that violates t-convexity of $Y$.

For the converse suppose that $Y = V (P \Box H)$ where $H \cong K_n$, $P$ is the only $v_1,v_k$-path in $G$ and all inner vertices of $P$ have degree 2 in $G$. Let $P = v_1,v_2, \ldots , v_k$. Since $H$ is a complete graph $T_{G\Box H}((v_i,h),(v_i,h'))=\{(v_i,h),(v_i,h')\} \subseteq Y$ for any $v_i \in P$ and $h,h'\in V(H)$. Thus let $(v_i,h),(v_j,h')$ be arbitrary vertices from $V(P \Box H)$, where $v_i,v_j \in P$, $i < j$. Let $(u,v)$ be an arbitrary vertex from $V(G \Box H)-Y$. We will prove that $N[(v_i,h)]-\{(u,v)\}$ separates $(u,v)$ from $(v_j,h')$ or $N[(v_j,h')]-\{(u,v)\}$ separates $(u,v)$ from $(v_i,h).$
Since $P$ is the only $v_1,v_k$-path in $G$, $d_G(u,v_i) \neq d_G(u,v_j).$ Thus we may assume without loss of generality that $d_G(u,v_i)<d_G(u,v_j).$ Let $R$ be a shortest $u,v_i$ path in $G$. Since $d_G(u,v_i)<d_G(u,v_j)$, $v_j\notin R$. As all inner vertices of $P$ are of degree 2, $R$ contains at least one vertex from $\{v_1,v_k\}$. Since $i<j$ and $v_i \in R$ but $v_j \notin R$ it is necessary that $v_1 \in R$ and $v_k \notin R$. Suppose that there exist $u,v_j$-path $P'$ in $G$ that does not contain $v_i$. Then $P'$ contains $v_k$, as all inner vertices of $P$ have degree 2 in $G$. Since $P$ is the only $v_1,v_k$-path in $G$, $P'$ does not contain $v_1.$ Therefore the walk obtained with concatenation of a shortest $v_1,u$-path and $u,v_k$-subpath of $P'$ can be reduced to $v_1,v_k$-path different from $P$, a contradiction. Thus $v_i$ is a cut vertex that separates $u$ from $v_j$ and consequently $^{v_i}\!H $ is a separating set in $G \Box H$ that separates $(u,v)$ from $(v_j,h')$. Since $H$ is a complete graph, $^{v_i}\!H  \subseteq N[(v_i,h)]-\{(u,v)\}$ and hence $N[(v_i,h)]-\{(u,v)\}$ separates $(u,v)$ from $(v_j,h')$ which by Proposition~\ref{p:separate} yields that $Y$ is t-convex. \qed

\begin{corollary}
Let $G\Box H$ be a connected Cartesian product, where $G$ and $H$ are not isomorphic to complete graph. Then $th(G \Box H)=2.$ 
\end{corollary}
\proof
Let $(g,h),(g',h')$ be arbitrary non-adjacent vertices in $G \Box H.$ Since $G$ and $H$ are not complete graphs, it follows from Theorem \ref{th:convex} that the only proper t-convex sets in $G \Box H$ are singletons. Thus the smallest t-convex set containing $(g,h),(g',h')$ is the whole vertex set $V(G \Box H).$ Therefore $th(G\Box H)=2.$ \qed

The above result holds also when one or both factors are complete graphs not isomorphic to $K_1$. In the rest of this section we will prove that it holds even more, that $tn(G\Box H)=2$ for any two non-trivial graphs $G$ and $H$.

\begin{lemma}\label{presecna}
Every non-complete connected graph $G$ has at least two distinct non-adjacent vertices that are not cut vertices.
\end{lemma}

\begin{proof}
Let $u,v$ be arbitrary vertices of $G$ with $d(u,v)=diam(G).$ Since $G$ is not a complete graph, $u$ and $v$ are not adjacent. We will prove that $u$ and $v$ are not cut vertices. To the contrary, let us assume that $u$ is a cut vertex.  Let $C_1$ be a connected  component of $G - \lbrace u \rbrace$ containing $v$ and let $w\in V(G)-\lbrace u,v \rbrace$ be a vertex in another component of $G - \lbrace u \rbrace$, say $C_2$. As $u$ is a cut vertex, every $v,w$-path contains $u$. Therefore $d(v,w)=d(v,u)+d(u,w)>diam(G)$, a contradiction. In the same way one can prove that $v$ is not a cut vertex, which completes the proof. \qed
\end{proof}

A well-known characterization of cut vertices says that a vertex $v$ of a connected graph $G$ is a cut vertex of $G$ if and only if there exist vertices $u,w \in V(G)-\lbrace v \rbrace$ such that  $v$ lies on every $u,w$-path in $G$. As we will use vertices, which are not cut vertices, very often, we characterize them in the following lemma:

\begin{lemma}\label{osn}
A vertex $v\in V(G)$ is not a cut vertex, if and only if for every pair of distinct vertices $u,w$ different from $v$, there is a $u,w$-path avoiding $v$.
\end{lemma}

In the following theorem we will prove that the toll interval between $(x,y)$ and $(u,v)$ in $G \Box H$ is a whole vertex set if $x,y$ are two non-adjacent vertices in $G$, $u,v$ are two non-adjacent vertices in $H$ and none of these four vertices is a cut vertex. Let us first explain why we need two non-adjacent vertices that are not cut vertices. In Figure~\ref{interval} we have the Cartesian product $P \Box Q$ of two paths isomorphic to $P_4$ where $P=a_1,a_2,a_3,a_4$ and $Q=b_1,b_2,b_3,b_4.$ It is easy to see that $(a_1,b_1) \notin T_{P\Box Q}((a_2,b_2),(a_4,b_4)).$ The problem is that there is no $a_1,a_4$-path avoiding $a_2$ which means that any tolled $(a_1,b_1),(a_4,b_4)$-walk will contain a neighbour of $(a_2,b_2).$  

\begin{figure}[h]
\centering
\begin{tikzpicture}[scale=1]
\foreach \x  in {0,1,2,3}
	{
	\draw (0,\x) -- (3,\x);
	}
\foreach \x  in {0,1,2,3}
	{
	\draw (\x,0) -- (\x,3);
	}

\foreach \x  in {0,1,2,3}
	\foreach \y  in {0,1,2,3}
		{
		\filldraw [fill=white, draw=black,thick] (\x,\y) circle (3pt);
		}

\draw(0,-0.5) -- (3,-0.5);
\draw(0,-1) node {$a_1$}; \draw(1,-1) node {$a_2$}; \draw(2,-1) node {$a_3$}; \draw(3,-1) node {$a_4$};

\draw(-0.5,0) -- (-0.5,3);
\draw(-1,0) node {$b_1$}; \draw(-1,1) node {$b_2$}; \draw(-1,2) node {$b_3$}; \draw(-1,3) node {$b_4$};

\foreach \x  in {0,1,2,3}
	\foreach \y in {1,2,3}
		{
		\filldraw [fill=brown, draw=black,thick] (\y,\x) circle (3pt);
		}		
\foreach  \x in {1,2,3}
{
	\filldraw [fill=brown, draw=black,thick] (0,\x) circle (3pt);
}		
\filldraw [fill=black, draw=black,thick] (1,1) circle (3pt);
\filldraw [fill=black, draw=black,thick] (3,3) circle (3pt);

\foreach \x in {0,1,2,3}
{
\filldraw [fill=white, draw=black,thick] (-0.5,\x) circle (3pt);
}	
\foreach \x in {0,1,2,3}
{
\filldraw [fill=white, draw=black,thick] (\x,-0.5) circle (3pt);
}		
\end{tikzpicture}
\caption{Vertex $(a_1,b_1) \notin T_{P\Box Q}((a_2,b_2),(a_4,b_4)).$}\label{interval}
\end{figure}
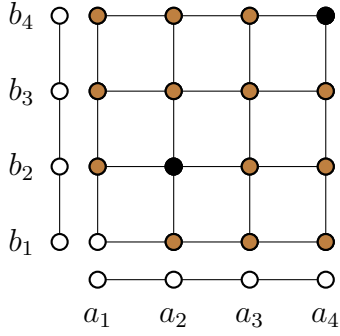

Sometimes we will refer to the shortest of all paths that does not contain vertex $u$ with the \textit{shortest path avoiding $u$.}

Let $G$ and $H$ be connected graphs, $x,y\in V(G)$, $h \in V(H)$ and $P=x,g_1,\ldots,g_k,y$ an $x,y$-path in $G$. Then by $P \times \lbrace h \rbrace$ we denote the path $(x,h), (g_1,h),(g_2,h),\ldots,(g_k,h),(y,h)$ between $(x,h)$ and $(y,h)$ in $G\Box H$. Similarly, a path $\lbrace u \rbrace \times Q$ is defined for $u \in V(G)$ and a path $Q$ between $u,v \in V(H)$.

Let $P=x,g_1,\ldots,g_k,y$ be an $x,y$-path and $Q=y,h_1,\ldots,h_l,z$ a $y,z$-path of a graph $G$. Then $P\cup Q$ denotes an $x,z$-walk that we get by concatenating paths $P$ and $Q$, i.e. a path $x,g_1,\ldots,g_k,y,h_1,\ldots,h_l,z$.

\begin{theorem}\label{kartezicni}
Let $G$ and $H$ be connected, non-complete graphs. Then $tn(G \Box H)=2$.
\end{theorem}

\begin{proof}
As $G$ and $H$ are not isomorphic to a complete graph, by Lemma \ref{presecna} there is a pair of (distinct) non-adjacent vertices in both graphs, $x, y \in V(G)$, such that  $xy \notin E(G)$ and $u,v \in V(H)$, such that $uv \notin E(H)$, and none of them is a cut vertex.
Then we claim that $$T_{G\Box H}((x,u),(y,v))=V(G\Box H).$$

First let $(g,h)\in V(G \Box H)$ such that $(g,h)\notin (\lbrace x,y \rbrace \times V(H)) \cup (V(G) \times \lbrace u,v\rbrace).$ Let $P_1$ be a shortest $x,g$-path avoiding $y$ and $Q_1$ a shortest $u,h$-path avoiding $v$. Such paths exist because of the Lemma \ref{osn} for $y$ and $v$. Then $R_1=(P_1\times \lbrace u \rbrace) \cup (\lbrace g \rbrace \times Q_1 )$ defines a $(x,u),(g,h)$-path, where the neighbour of $(x,u)$ on $R_1$ is the only neighbour of $(x,u)$ (because $P_1$ is shortest of such paths that avoid $y$) and no vertex from $R_1$ is adjacent to $(y,v)$ as $u$ is not adjacent to $v$ and  $Q_1$ does not contain $y$.
	
Similarly let $P_2$ be a shortest $g,y$-path avoiding $x$ and $Q_2$ a shortest $h,v$-path avoiding $u$. Then $R_2=(P_2\times \lbrace h \rbrace) \cup (\lbrace y \rbrace \times Q_2 )$ defines a $(g,h),(y,v)$-path, where a neighbour of $(y,v)$ on $R_2$ is the only neighbour of $(y,v)$ on this path (because $Q_2$ is shortest of such paths that avoid $u$) and no vertex from $R_2$ is adjacent to $(x,u)$ as $x$ is not adjacent to $y$ and $P_2$ does not contain $x$.

Thus $R_1 \cup R_2$ is a tolled walk between $(x,u)$ and $(y,v)$ that contains $(g,h).$

	Let now $g=x$. Similar arguments hold if $g=y$. 
	First let $h=v$. Let $P$ be the shortest $x,y$-path in $G$ and $Q$ the shortest $u,v$-path in $H$. Note that every shortest path is a tolled walk. Then $(\lbrace x \rbrace \times Q) \cup (P \times \lbrace v \rbrace)$ is a desired tolled walk.	 
	 Finally let $h \neq v$. Let $P$ be a shortest $x,y$-path in $G$, $Q_1$ a shortest $u,h$-path in $H$, which avoids $v$ and $Q_2$ a shortest $h,v$-path that avoids $u$ in $H$. Then $(\lbrace x \rbrace \times Q_1) \cup (P\times \lbrace h \rbrace) \cup (\lbrace y \rbrace \times Q_2)$ again is a tolled $(x,u),(y,v)$-walk containing a vertex $(g,h)$. By the symmetry, there also exists a tolled $(x,u),(y,v)$-walk  containing $(g,h)$, when $h=u$ or $h=v$. \qed
\end{proof}

Theorem \ref{kartezicni} excludes cases when at least one factor is a complete graph. 

\begin{proposition}\label{Gcomplete}
Let $n \geq 2$ and let $G$ be a connected non-complete graph. Then $tn(G \Box K_m)=2$.
\end{proposition}

\begin{proof}
Let $x,y$ be arbitrary non-adjacent vertices in $G$ that are not cut vertices (the existence follows from Lemma~\ref{presecna}) and let $u,v$ be two different vertices from $K_n$. Let $(g,h)\in V(G\Box H)-\{(x,u),(y,v)\}$.
If $g=x$, let $P$ be a shortest $x,y$-path in $G$. Then $(x,u) \cup (x,h) \cup (P \times \{h\}) \cup (y,v)$ is a tolled $(x,u),(y,v)$-walk containing $(g,h).$ If $g=y$, a tolled $(x,u),(y,v)$-walk containing $(g,h)$ is obtained in a similar way. Now let $P_1$ be a shortest $x,g$-path avoiding $y$ and $P_2$ a shortest $g,y$-path avoiding $x$.
If $h=u$ (similar arguments give a desired result when $h=v$), then $(P_1 \times \{u\}) \cup (g,v) \cup (P_2 \times \{v\})$ is a tolled $(x,u),(y,v)$-walk containing $(g,h)$. Finally let $g \neq x,y$ and $h\neq u,v$. Then $(P_1\times \{u\}) \cup (g,v) \cup (P_2 \times \{h\}) \cup (y,v)$ is a tolled $(x,u),(y,v)$-walk containing $(g,h)$.  \qed
\end{proof}

\begin{proposition}\label{complete}
Let $n,m\geq 2.$ Then $tn(K_n \Box K_m)=2$.
\end{proposition}

\begin{proof}
Let $x$ and  $y$ be two different vertices in $V(K_n)$ and $u$ and $v$ two different vertices in $V(K_m)$. Then observe that $(x,u)(y,v)\notin E(K_n\Box K_m)$ and that $T_{G\Box H}((x,u),(y,v))=V(G\Box H).$ To prove the second, let $(g,h)\in V(K_n\Box  K_m)$. Assume first that $g=x$. If $h=v$, then $(x,u),(x,v),(y,v)$ is a tolled walk between $(x,u)$ and $(y,v)$ containing $(g,h)$. Otherwise $(x,u),(x,h),(y,h),(y,v)$ is a desired tolled walk. Similar construction works if $g=y$ or $h \in \lbrace u,v \rbrace$. Now let $g\notin\lbrace x, y\rbrace$ and $h \notin \lbrace u,v \rbrace$. Then $(x,u),(g,u),(g,h),(y,h),(y,v)$ is a tolled walk  between $(x,u)$ and $(y,v)$ containing the vertex $(g,h)$.\qed
\end{proof}

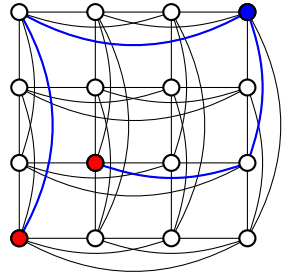
\begin{figure}[h]
\centering
\begin{tikzpicture}[scale=1]
\foreach \x  in {0,1,2,3}
	{
	\draw (\x,0) -- (\x,3);
	\draw (0,\x) -- (3,\x);
	\draw (0,\x) to [bend right=20] (2,\x);
	\draw (1,\x) to [bend right=20] (3,\x);
	\draw (0,\x) to [bend right=30] (3,\x);
	\draw (\x,0) to [bend right=20] (\x,2);
	\draw (\x,1) to [bend right=20] (\x,3);
	\draw (\x,0) to [bend right=30] (\x,3);
	}

\draw[color=blue, thick] (0,0) to [bend right=30] (0,3)  to [bend right=30] (3,3)  to [bend left=20] (3,1)  to [bend left=20] (1,1) ;

\foreach \x  in {0,1,2,3}
	\foreach \y  in {0,1,2,3}
		{
		\filldraw [fill=white, draw=black,thick] (\x,\y) circle (3pt);
		}
\filldraw [fill=red, draw=black,thick] (0,0) circle (3pt);
\filldraw [fill=red, draw=black,thick] (1,1) circle (3pt);
\filldraw [fill=blue, draw=black,thick] (3,3) circle (3pt);
\end{tikzpicture}
\caption{A tolled walk in $K_4 \Box K_4$ containing a blue vertex.\label{polni}}
\end{figure}

\begin{corollary}\label{tollN}
Let $G$ and $H$ be arbitrary connected non-trivial graphs. Then $tn(G \Box H)=2.$
\end{corollary}

Since $th(G) \leq tn(G)$ for any graph $G$ Corollary~\ref{tollN} implies the following result.

\begin{corollary}
Let $G$ and $H$ be arbitrary connected non-trivial graphs. Then $th(G \Box H)=2.$
\end{corollary}


\section{The lexicographic product}\label{s:Lex}

A characterization of t-convex sets of $G \circ H$ from~\cite{abg} is again a natural way to obtain t-hull number of $G \circ H.$ A set $Y$, where $Y\subset V(G\circ H)$, is said to be \emph{non-extreme complete} if $^g\!H \cap Y =\, ^g\!H$ holds for all non-extreme vertices $g$ of $p_G(Y)$.

\begin{theorem}\cite[Theorem 5.1]{abg}\label{lexOld}
Let $G\circ H$ be a non-trivial, connected
lexicographic product. A proper subset $Y$ of $V(G\circ H)$, which
does not induce a complete graph, is t-convex if and only if the
following conditions hold:
\item[(a)] $p_G(Y)$ is t-convex in $G$,
\item[(b)] $Y$ is non-extreme complete, and
\item[(c)] $H$ is complete.
\end{theorem} 

\begin{corollary}
Let $G$ and $H$ be non-trivial connected graphs where $H$ is not complete. Then $th(G \circ H)=2.$
\end{corollary}
\begin{proof}
Since $H$ is not complete Theorem~\ref{lexOld} implies, that the only t-convex set containing two different non adjacent vertices from $V(G \circ H)$ is the whole graph. Thus $th(G\circ H)=2.$ \qed
\end{proof}

Before we continue with the investigation of the toll number and t-hull number of the lexicographic product we give a formula for the toll interval between two vertices in the lexicographic product of two graphs.

\begin{lemma}\label{interval}
Let $G$ and $H$ be two non-trivial graphs. Then $T_{G \circ H}((g,h),(g',h'))=((T_G(g,g')-\{g,g'\})\times V(H) ) \cup \{(g,h),(g',h')\}$ for every $g\neq g'.$
\end{lemma}
\begin{proof}
First, if $gg'\in E(G)$ then $T_{G \circ H}((g,h),(g',h'))=\{(g,h),(g',h')\}$ and the proof is complete. 

Thus let $g,g'$ be two non-adjacent vertices of $G$. Let $y \in V(H)$, $x \in T_G(g,g')-\{g,g'\}$ and let $W=g,g_1,\ldots , g_k,g'$ be a tolled $g,g'$-walk that contains $x=g_i$ for some $i \in \lbrace 1,2,\ldots,k\rbrace$. Then $$(g,h),(g_1,h),\ldots (g_{i-1},h),(x,y),(g_{i+1},h'),\ldots, (g_k,h'),(g',h')$$ is a tolled $(g,h),(g',h')$-walk that contains $(x,y).$  

Let $y \in V(H)-\{h\}.$ Then $N[(g,h)]-\{(g,y)\}$ separate $(g,y)$ from $(g',h')$. Thus using Lemma~\ref{l:abg} we get $(g,y) \notin T_{G \circ H}((g,h),(g',h'))$. In an analogue way we prove that $(g',y) \notin T_{G \circ H}((g,h),(g',h'))$ for any $y \in V(H)-\{h'\}.$

Finally let $x \notin T_G(g,g')$. Then, using Lemma~\ref{l:abg}, we may without loss of generality assume that $N[g]-\{x\}$ separate $x$ from $g'$. Therefore $N[(g,h)]-\{(x,y)\}$ separate $(x,y)$ from $(g',h')$ and consequently $(x,y) \notin T_{G \circ H}((g,h),(g',h')).$   \qed

\end{proof}

If $g=g'$ then we get the following obvious remark.

\begin{remark}\label{sloj}
Let $G$ and $H$ be two non-trivial graphs. Then 
\[
    T_{G \circ H}((g,h),(g,h'))= 
\begin{cases}
    (N_G(g)\times V(H)) \cup \{(g,x);~x \in T_H(h,h')\},& \text{if } h' \notin N_H[h];\\
    \{g\} \times T_H(h,h'),              & \text{if } h' \in N_H[h].
\end{cases}
\]
\end{remark}

\begin{proposition}\label{p:hull}
Let $G$ be a non-trivial graph. Then $n \cdot |Ext(G)| \leq th(G \circ K_n) \leq n\cdot th(G).$
\end{proposition}

\begin{proof}
Let $D$ be a minimum t-hull set of $G$. Then $S=\{(g,h);~g \in D, h \in V(K_n)\}$ is a t-hull set of $G \circ K_n.$ Indeed, if $x \in V(G)-D$ then $x \in T^k(D)$ for some $k \in \mathbb N$ and consequently using Lemma~\ref{interval}, $(x,y) \in T^k(D \times V(K_n))$ for every $y \in V(K_n).$ Thus $th(G \circ K_n) \leq n \cdot th(G).$

For the second inequality, let $D$ be the set of extreme vertices of $G$. We will prove that the vertices in $D \times V(K_n)$ are extreme. Let $(g,h) \in D \times V(K_n)$. We will prove that $(G\circ K_n)-\{(g,h)\}$ is t-convex, i.e.~ $(g,h) \notin T((x,y),(x',y'))$ for any $(x,y),(x',y') \in V(G \circ K_n)$ different from $(g,h)$. Suppose first that $g=x=x'$. Then $(g,h),(x,y),(x',y')$ induces a clique in $G \circ K_n$ and thus $(g,h) \notin T((x,y),(x',y'))$. If exactly one vertex from $\{x,x'\}$ is equal to $g$, then Lemma~\ref{interval} implies that $(g,h) \notin T((x,y),(x',y'))$. Finally suppose that $g \notin \lbrace x,x'\rbrace.$ For the purpose of contradiction assume that $(g,h) \in T((x,y),(x',y'))$, then it follows from Lemma~\ref{interval}, that $g \in T_G(x,x')$, which contradicts the fact that $g$ is an extreme vertex of $G$. Therefore $D \times V(K_n) \subseteq Ext(G\circ K_n)$ which implies, $n \cdot |Ext(G)| \leq |Ext(G \circ K_n)| \leq th(G \circ K_n).$  \qed 
\end{proof}

In the rest of the section we will focus on the toll number of the lexicographic product of graphs. First we will consider the lexicographic product of $G$ and a complete graph.

\begin{proposition}
Let $G$ be a non-trivial graph. Then $n \cdot |Ext(G)| \leq tn(G \circ K_n) \leq n\cdot tn(G).$
\end{proposition}
\begin{proof}
Let $D$ be a minimum toll set of $G$. Then $S=\{(g,h);~g \in D, h \in V(K_n)\}$ is a toll set of $G \circ K_n.$ Indeed, if $x \in V(G)-D$ then there exist $g, g' \in D$ such that $x \in T_G(g,g')$. Then it follows from Lemma~\ref{interval} that $(x,y) \in T_{G \circ K_n}((g,y),(g',y))$ for every $y \in V(K_n).$ Thus $tn(G \circ K_n) \leq n\cdot tn(G).$

The second inequality follows directly from Proposition~\ref{p:hull}, as $n \cdot |Ext(G)| \leq th(G \circ K_n) \leq tn(G \circ K_n).$ \qed 
\end{proof}

\begin{corollary}
Let $G$ be a non-trivial extreme complete graph. Then $tn(G \circ K_n)=n \cdot tn(G).$
\end{corollary}

On the other hand there is an infinite family of graphs where the upper bound is not sharp. For example, $tn(C_m)=2$ for every $m >3.$ But any four vertices of one $C_m$-layer induce a toll set of $C_m \circ K_n$. Even more, if $m \geq 6$ any three pairwise non-adjacent vertices of one $C_m$-layer induce a toll set of $C_m \circ K_n$. Thus, $tn(C_m \circ K_n) \leq 4, $ and for $m \geq 6$, $tn(C_m \circ K_n) \leq 3$ which is for $n >> 4$, much less than $n \cdot tn(C_m)=2n.$ 

\begin{theorem}\label{meja3}
Let $G$ and $H$ be arbitrary connected non-trivial graphs where $H$ is not a complete graph. Then $tn(G \circ H) \leq 3 \cdot tn(G)$.
\end{theorem}
\begin{proof}
Let $D$ be a minimum toll set of $G$. We will construct a toll set $S$ in $G \circ H$ of size $3|D|.$ For every vertex $g \in D$ we put in $S$ an arbitrary vertex from $^g\!H.$ Then the toll closure of $S$ contains all vertices of $G \circ H$ except some vertices in $^g\!H$ for $g \in D$. But those vertices can be covered with the intervals between the vertices in the neighboring layers. Thus, for every $g \in D$, we add to $S$ $(g',h),(g',h')$, where $g'$ is a neighbour of $g$ in $G$ and $h,h'$ are arbitrary different, non-adjacent vertices of $H$. \qed
\end{proof}

The bound from Theorem~\ref{meja3} is sharp. We will show that later. First, examples of families of graphs, for which the bound in Theorem~\ref{meja3} is not achieved, appear in the following lemmas and theorem.

\begin{lemma}\label{univerzalno}
Let $G$ and $H$ be arbitrary non-trivial graphs, where $G$ has an universal vertex and $H$ is not a complete graph. Then $tn(G \circ H) \leq 4.$
\end{lemma}
\begin{proof}
In this case we will construct a toll set of size 4. Let $h_1,h_2$ be arbitrary different, non-adjacent vertices from $H$ and $g$ a universal vertex of $G$. Then $T_{G \circ H}((g,h_1),(g,h_2))$ covers all vertices of $V(G \circ H)$ with the exception of some vertices in $^g\!H.$ Since for any neighbour $g'$ of $g$, the interval $T_{G \circ H}((g',h_1),(g',h_2))$ contains all vertices from $^g\!H$, the proof is completed. \qed    
\end{proof}

\begin{lemma}\label{2neigh}
Let $G$ and $H$ be arbitrary non-trivial graphs, where $G$ has vertices $u$ and $v$ with $N(u) \cup N(v)=V(G)$ and $H$ is not a complete graph. Then $tn(G \circ H) \leq 4.$
\end{lemma}
\begin{proof}
First note that $u$ and $v$ are adjacent, as $N(u) \cup N(v)=V(G)$. Let $h_1,h_2$ be arbitrary different, non-adjacent vertices from $H$. Then $S=\{ (u,h_1),(u,h_2),(v,h_1),(v,h_2)\}$ is a toll set of $V(G \circ H)$. \qed
\end{proof}


In the following theorem we will characterize pairs of graphs $(G,H)$ with $tn(G \circ H)=2.$

\begin{theorem}
Let $G$ and $H$ be arbitrary non-trivial graphs where $H$ is not isomorphic to $K_2$. Then $tn(G \circ H)=2$ if and only if $G$ has an universal vertex and $tn(H)=2.$
\end{theorem} 
\begin{proof}
Suppose $g$ is an universal vertex of $G$ and $\{h_1,h_2\}$ toll set of $H$. Then $\{(g,h_1),(g,h_2)\}$ is a toll set of $G \circ H.$ Indeed, $(g,x) \in T_{G \circ H}((g,h_1),(g,h_2))$ for any $x \in V(H)$, as for any tolled $h_1,h_2$-walk $W$ containing $x$, $\{g\} \times W$ is a tolled $(g,h_1),(g,h_2)$-walk containing $(g,x)$. If $x$ is an arbitrary vertex from $H$ and $g'$ an arbitrary vertex from $G-\{g\}$ then $(g,h_1),(g',x),(g,h_2)$ is a tolled $(g,h_1),(g,h_2)$-walk containing $(g',x).$

For the converse suppose that $tn(G \circ H)=2$ and let $D=\{(g,h),(g',h')\}$ be a toll set of $G \circ H.$ Lemma~\ref{interval} implies that $g=g',$ otherwise $(g,h_1) \notin T_{G \circ H}((g,h),(g',h'))$ for any $h_1\neq h.$ Since $T_{G \circ H}((g,h),(g,h'))=V(G \circ H)$ it follows from Remark~\ref{sloj} that $g$ is a universal vertex of $G$ and $tn(H)=2$, which completes the proof. \qed
\end{proof}

On the other hand there is an infinite family of graphs with $tn(G \circ H)=3\cdot tn(G).$ For example, it is easy to see that $tn(P_n \circ K_{1,m})=3\cdot tn(P_n)=6$ for every $n \geq 4, m \geq 3.$ Now we focus in finding a characterization of graphs for which the bound from Theorem~\ref{meja3} is tight. 

From the proof of Theorem~\ref{meja3} it follows that whenever two vertices $g,g'$ from a minimum toll set $S$ of $G$ have a common neighbour $g''$ we can reduce the bound for 2. Instead of adding to $S$ two vertices from one neighboring layer of $g$ and another two from different neighboring layer of $g'$, add only two vertices from $^{g''}\!H$. The discussion of this paragraph gives the following necessary condition for $tn(G \circ H)=3\cdot tn(G).$

\begin{lemma}\label{l:2packing}
Let $G$ and $H$ be arbitrary connected non-trivial graphs where $H$ is not a complete graph. If $tn(G \circ H)=3\cdot tn(G)$, then for every minimum toll set $D$ of $G$ it holds that $N(u) \cap N(v) = \emptyset$ for any two different vertices $u,v \in D$.
\end{lemma}

\begin{proof}
Suppose that there exists a minimum toll set $D$ with different vertices $x,y \in D$ such that $N(x) \cap N(y) \neq \emptyset.$ Let $u \in N(x) \cap N(y)$. Since the vertices of $H$-layers $^{x}\!H$ and $^{y}\!H $ are contained in the toll interval between $(u,h)$ and $(u,h')$ for arbitrary non-adjacent vertices $h,h'\in V(H)$ the arguments of the proof of Theorem~\ref{meja3} imply that $tn(G \circ H) \leq 3\cdot tn(G)-2,$ a contradiction.   \qed
\end{proof}

\begin{corollary}\label{notComplete}
Let $G$ and $H$ be arbitrary connected non-trivial graphs where $H$ is not a complete graph. If $tn(G \circ H)=3\cdot tn(G)$, then $G$ is not a complete graph.
\end{corollary}
\begin{proof}
Suppose that $G$ is a complete graph isomorphic to $K_n$. Since $G$ has an universal vertex Lemma~\ref{univerzalno} implies that $tn(G \circ H) \leq 4 < 3 \cdot tn(G) = 3n$, a contradiction. \qed
\end{proof}

A set $S \subseteq V(G)$ is a {\it 2-packing} if $N[u] \cap N[v] = \emptyset$ for any $u,v \in S.$

\begin{lemma}\label{tn2}
Let $G$ and $H$ be arbitrary connected non-trivial graphs where $H$ is not a complete graph. If $tn(G \circ H)=3\cdot tn(G)$, then $tn(G)=2.$
\end{lemma}
\begin{proof}
Suppose that $tn(G) > 2$ and let $D=\{u_1,\ldots , u_k\}$ be a minimum toll set of $G$. We will prove that there exist three different indices $i,j,l$ from $\{1,\ldots , k\}$ such that $u_i \in T_G(u_j,u_l).$ Suppose first that there exist three pairwise non-adjacent vertices $u_1,u_2,u_3$ in $D$. Therefore using Lemma~\ref{l:2packing} the set $\{u_1,u_2,u_3\}$ is a 2-packing. Suppose that $u_3 \notin T_G(u_1,u_2)$ (otherwise the desired assertion is already proved). Then without loss of generality we may assume using Lemma~\ref{l:abg} that $N[u_1]$ separate $u_3$ from $u_2.$ Since $G$ is connected there exist $u_2,u_3$-path $P$ that contains a vertex from $N[u_1]$. If $P$ contains $u_1$ then $u_1\in T_G(u_2,u_3)$. Otherwise we construct the walk $W$ from $P$ in such way that we add $u_1$ after the first vertex from $N[u_1]$ that appeared on $P$. Since $d(u_1,u_3),d(u_1,u_2) \geq 2$, $W$ is a tolled walk and thus $u_1\in T_G(u_2,u_3).$

Suppose now that $u_1$ and $u_3$ are adjacent but $u_2$ is not adjacent to $u_1,u_3$. (There exist three vertices from $D$ that are not pairwise adjacent, as $G$ is not a complete graph by Corollary~\ref{notComplete}. If the three vertices induce a path, then the desired assertion is already proved.) It follows from Lemma~\ref{l:2packing} that $N(u_1) \cap N(u_3) = N[u_1] \cap N[u_2]= N[u_2] \cap N[u_3] = \emptyset.$ Suppose that $u_3 \notin T_G(u_1,u_2)$ (otherwise the desired assertion is already proved). Then it follows from Lemma~\ref{l:abg}, that $N[u_1]-\{u_3\}$ separates $u_2$ and $u_3$ or $N[u_2]$ separates $u_1$ and $u_3$. The last assertion is not possible as $u_1$ and $u_3$ are adjacent. Thus $N[u_1]-\{u_3\}$ separates $u_2$ and $u_3$. Since $G$ is connected, there exists an $u_2,u_3$-path $P$ that intersect $N[u_1]-\{u_3\}$. Let $w$ be the first vertex on $P$ that is from $N[u_1]$ and let $R$ be the $u_2,w$-subpath of $P$. Then $R,u_1,u_3$ is a tolled $u_2,u_3$-walk containing $u_1$. Note that vertices from $R-\{w\}$ are not adjacent to $u_3$, otherwise $N[u_1]-\{u_3\}$ would not separate $u_2$ and $u_3$.

Thus we may assume without loss of generality that $u_1 \in T_G(u_2,u_3)$. Now we will construct a toll set $S$ of $G \circ H$ with the size less than $3\cdot tn(G)$. The construction goes in the same way as in the proof of Theorem~\ref{meja3}. 
For any vertex $g \in D$ we put in $S$ an arbitrary vertex from $^g\!H.$ Then the toll closure of $S$ contain all vertices of $G \circ H$ except some vertices in $^g\!H$ for $g \in D-\{u_1\}$. (Note that the vertices from $^{u_1}\!H$ lie on the toll interval between $(u_2,h)$ and $(u_3,h)$ for some $h \in V(H)$.)  But those vertices can be covered with the intervals between the vertices in the neighboring layers. Thus, for ever $g \in D-\{u_1\}$, we add to $S$ $(g',h),(g',h')$, where $g'$ is a neighbour of $g$ in $G$ and $h,h'$ are arbitrary different, non-adjacent vertices of $H$. Thus $S$ is a toll set of $G \circ H$ of size $3 \cdot tn(G)-2$, which is a contradiction.  \qed
\end{proof}

\begin{lemma}\label{tnHvsaj3}
Let $G$ and $H$ be arbitrary connected non-trivial graphs where $H$ is not a complete graph. If $tn(G \circ H)=3\cdot tn(G)$, then $tn(H)>2.$
\end{lemma}
\begin{proof}
Suppose that $tn(H) \leq 2$. Then $tn(G \circ H) \leq 2 \cdot tn(G)$, a contradiction. \qed
\end{proof}

\begin{lemma}\label{pogoj}
Let $G$ and $H$ be arbitrary connected non-trivial graphs where $H$ is not a complete graph and $G$ is not isomorphic to $K_2$. Suppose that the following conditions are satisfied:
\begin{enumerate}
\item $tn(G)=2=|Ext(G)|$ with $Ext(G)=\{u,v\}$;
\item $tn(H) > 2$;
\item $N_G(x) \cup N_G(y) \neq V(G)$ for any $x, y \in V(G)$;
\item $d_G(u,v) \geq 3$ and if $d_G(u,v)=3$ then $\forall z \in N_G(u') \cup N_G(v')$, $\exists x \notin N_G(u') \cup N_G(v')$ such that $x \notin T_G(u',z) \cup T_G(v',z)$, where $u'\in N_G(u)$ and $v' \in N_G(v)$ are arbitrary adjacent vertices of $G$.   
\end{enumerate}
Then $tn(G \circ H) = 3 \cdot tn(G)$. 
\end{lemma}

\begin{proof}
As $tn(G)=2$ and all extreme vertices of $G$ are contained in any toll set of $G$, $\{u,v\}$ is a minimum toll set of $G$. Let $S$ be a minimum toll set of $G \circ H$. We will prove that $|S| \geq 3\cdot tn(G) =6$ (we already know from Theorem~\ref{meja3} that $|S| \leq 6$).

Since $u$ is an extreme vertex in $G$ and $^u\!H$ is contained in the toll closure of $S$, it follows from Lemma~\ref{interval} that $S$ either contains vertices $\{(u,h):~h \in  D,~D \textrm{ is a toll set of }H\}$ or $S$ contains $(u',h),(u',h')$, where $u'$ is a neighbor of $u$ in $G$ and $h$ and $h'$ are arbitrary non-adjacent vertices from $H$. Since $v$ is also an extreme vertex of $G$, we get that $S$ either contains vertices $\{(v,h):~h \in  D,~D \textrm{ is a toll set of }H\}$ or $S$ contains $(v',h),(v',h')$, where $v'$ is a neighbor of $v$ in $G$ and $h$ and $h'$ are arbitrary non-adjacent vertices from $H$. We will distinguish 3 cases. If $S$ contains $\{(u,h):~h \in  D,~D \textrm{ is a toll set of }H\} \cup \{(v,h):~h \in  D,~D \textrm{ is a toll set of }H\}$ then $|S| \geq 6$, as $tn(H) > 2$ and the proof is completed. Suppose now that $S$ contains $S'=\{(u,h):~h \in  D,~D \textrm{ is a toll set of }H\}$ and $S$ contains $(v',h),(v',h')$, where $v'$ is a neighbor of $v$ in $G$ and $h$ and $h'$ are arbitrary non-adjacent vertices from $H$ (note that the same follows if we choose toll set of $H$ in the $^v\!H$-layer and two non-adjacent vertices from a neighboring layer of $^u\!H$). Since $tn(H) > 2$ it follows from Lemma~\ref{interval}, that at least one vertex from $H$-layer $^{v'}\!H$ is not contained in the toll closure of $S' \cup \{(v',h),(v',h')\}$. Thus $|S| \geq 6$ and the proof is completed.

Suppose now that $S'=\{(u',h_1),(u',h_2),(v',h_3),(v',h_4)\},$ where $u'$ is a neighbor of $u$ and $v'$ is a neighbor of $v$ in $G$, $h_1h_2,h_3h_4\notin E(H)$ and $h_1 \neq h_2,h_3\neq h_4.$ Since $d(u,v) \geq 3$, $u \neq v$. Suppose first that $u'$ and $v'$ are not adjacent.
Then $u'\neq v'$ and $u$ is not adjacent to $v$ in $G$. Then $^{u'}\!H$, $^{v'}\!H$ are not contained in the toll closure of $S'$. We will prove that $S'$ together with an arbitrary vertex of $G \circ H$ is not a toll set of $G \circ H$. Moreover we will prove that $(^{u'}\!H) \cup (^{v'}\!H)$ is not contained in the toll closure of $S' \cup \{(g,h)\}$ for an arbitrary $(g,h) \in V(G \circ H).$ For this purpose we have to prove that there is no $x \in V(G)$ such that $u' \in T_G(x,v')$ and $v' \in T_G(x,u').$ For the purpose of contradiction suppose that there exists $x \in V(G)$ such that $u' \in T_G(x,v')$ and $v' \in T_G(x,u').$ Let $W$ be a tolled $x,v'$-walk that contains $u'$. Since $u$ is extreme, $u \notin W$ and $u$ is not adjacent to $x$. Indeed, if $u$ is adjacent to $x$, then $u'$ is adjacent to $x$, since $u$ is simplicial, a contradiction. Thus $xWu',u,u'Wv'$ is a tolled walk containing $u$, which contradicts the fact that $u$ is extreme. 

Finally suppose that $u'$ is adjacent to $v'$. We will again prove that $S' \cup \{(g,h)\}$ is not a toll set of $G \circ H$ for any $(g,h) \in V(G \circ H)$. Since $u'$ and $v'$ are adjacent it follow from Lemma~\ref{interval} and Remark~\ref{sloj} that the toll closure of $S'$ contains $(^{x'}\!H)$ for any $x' \in N(u') \cup N(v')$. As condition 3 of this lemma implies that $N(u') \cup N(v') \neq V(G)$, $S'$ contains at least one vertex $(g,h) \in V(G \circ H)$ with $g \neq u',v'$. If $g \notin N(u') \cup N(v')$, then $^{g}\!H$ is not contained in the toll closure of $S' \cup \{(g,h)\}$ and hence $|S| \geq 6$. If $g \in N(u') \cup N(v')$, then it follows from the condition 4 of the lemma that there exists $x \notin N(u') \cup N(v')$ such that $x \notin T_G(u',g) \cup T_G(v',g)$. Thus $^{x}\!H$ is not contained in the toll closure of $S' \cup \{(g,h)\}$ and hence $|S| \geq 6$. \qed

\end{proof}

\begin{theorem}
Let $G$ and $H$ be arbitrary connected non-trivial graphs where $H$ is not a complete graph and $G$ is 
a graph with $|Ext(G)|\geq 2$ not isomorphic to $K_2$. 
Then $tn(G \circ H)= 3 \cdot tn(G)$ if and only if the following conditions hold:
\begin{enumerate}
\item $tn(G)=2$ with $Ext(G)=\{u,v\}$;
\item $tn(H) > 2$;
\item $N_G(x) \cup N_G(y) \neq V(G)$ for any $x, y \in V(G)$;
\item $d_G(u,v) \geq 3$ and if $d_G(u,v)=3$ then $\forall z \in N_G(u') \cup N_G(v')$, $\exists x \notin N_G(u') \cup N_G(v')$ such that $x \notin T_G(u',z) \cup T_G(v',z)$, where $u'\in N_G(u)$ and $v' \in N_G(v)$ are arbitrary adjacent vertices of $G$.    
\end{enumerate}
\end{theorem}

\begin{proof}
If the four conditions are satisfied, then Lemma~\ref{pogoj} implies that $tn(G \circ H)=3\cdot tn(G)$.

For the converse suppose that $tn(G \circ H)= 3 \cdot tn(G)$. Then the conditions 1--3 follow from lemmas Lemma~\ref{tn2}, Lemma~\ref{tnHvsaj3}, Lemma~\ref{2neigh}. We have already explained that whenever two vertices from minimum toll set have a common neighbor, $tn(G \circ H) < 3 \cdot tn(G)$. Thus $d(u,v) \geq 3$. Suppose that $d(u,v)=3$ and that there exists $z \in N(u') \cup N(v')$ such that for any $x \notin N(u') \cup N(v')$, $x \in T_G(u',z) \cup T_G(v',z)$, where $u,u',v',v$ is a $u,v$-path of length 3. Let $h, h'$ be arbitrary nonadjacent vertices from $H$. Then $\{(u',h_1),(u',h_2),(v',h_1),(v',h_2),(z,h_1)\}$ is a toll set of size less than $3 \cdot tn(G)$, a contradiction.
\qed
\end{proof}

The characterization of graphs with $tn(G \circ H)= 3 \cdot tn(G)$ is incomplete in case when $|Ext(G)|\in \lbrace 0,1 \rbrace$ (note that $|Ext(G)|\leq tn(G)=2$). Figure \ref{pr1} shows examples of graphs $G$ and $G'$ with $tn(G)=tn(G')=2$, $Ext(G)=Ext(G')=\emptyset$, but $tn(G \circ H)=3<3\cdot tn(G)$ and $tn(G' \circ H)=6=3\cdot tn(G')$ for any graph $H$ with $tn(H)>2$.

\begin{figure}[h]
\centering
\begin{tikzpicture}[scale=1]

	\draw (0,0) -- (4,0);
	\draw (2,0) -- (2,2);
	\node[] at (0,2) {$G$};

	\foreach \x in {0,1,2,3,4}
		{
		\filldraw [fill=white, draw=black,thick] (\x,0) circle (3pt);
		}
		\filldraw [fill=white, draw=black,thick] (2,1) circle (3pt);
		\filldraw [fill=white, draw=black,thick] (2,2) circle (3pt);
\end{tikzpicture}
\hspace{1cm}
\begin{tikzpicture}[scale=1]

	\draw (0,0) -- (2,0);
	\draw (0,2) -- (2,2);
	\draw (0,0) -- (0,2);
	\draw (2,0) -- (2,2);
	
	\draw (0,1) -- (2,0);
	\draw (0,2) -- (1,0);
	\draw (0,2) -- (2,1);
	\draw (1,2) -- (2,0);

		\node[] at (-1,2) {$G'$};

	\foreach \x in {0,1,2}
		{
		\filldraw [fill=white, draw=black,thick] (\x,0) circle (3pt);
		\filldraw [fill=white, draw=black,thick] (\x,2) circle (3pt);
		}
		\filldraw [fill=white, draw=black,thick] (0,1) circle (3pt);
		\filldraw [fill=white, draw=black,thick] (2,1) circle (3pt);
\end{tikzpicture}

\caption{Graphs $G$ and $G'$.\label{pr1}}
\end{figure}
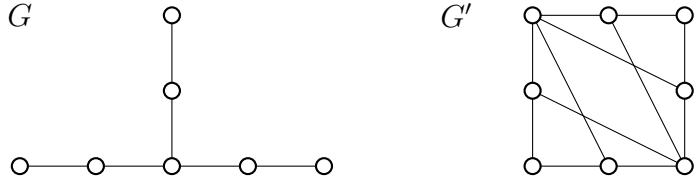

\begin{question}
Is there a graph $G$ with $|Ext(G)|=1$ with respect to toll convexity?
\end{question}

\begin{question}
Can you give a complete characterization of graphs with $tn(G \circ H)=3\cdot tn(G)$?
\end{question}

Finally we give the exact toll number for the lexicographic product of two connected graphs $G$ and $H$, where $H$ is not a complete graph. We use the so-called toll-dominating triple, which was introduced in~\cite{bkt} in terms of the geodetic number of a graph.

Let $A,B,C$ be pairwise disjoint subsets of a vertex set of a graph $G$. Then $(A,B,C)$ is a {\it{toll-dominating triple}} if for any $x \in V(G)-C$ there exist $u,v \in A \cup B \cup C$ with $x \in T_G(u,v)-\{u,v\}$ or there exists $w \in B \cup C$ with $x \in N_G(w).$ 

\begin{lemma}\label{012}
Let $S$ be a minimum toll set of $G \circ H.$ Then $|S \cap$ $^g\!H| \in \{0,1,2,tn(H)\}$ for any $g \in V(G).$
\end{lemma}
\begin{proof}
Suppose that there exists $g \in V(G)$ with $2 < |S \cap$ $^g\!H| < tn(H)$ and let $S_g=S \cap$ $^g\!H=\{(g,h_1),\ldots , (g,h_k)\}.$ Since $k < tn(H)$ there exists $(g,h) \in$ $^g\!H$ such that $(g,h)$ is not in the toll closure of $S_g.$ As $S$ is a toll set of $G \circ H$, there exist $(g_1,h_1),(g_2,h_2) \in S-S_g$ such that $(g,h) \in T_G((g_1,h_1),(g_2,h_2))$. Therefore it follows from Lemma~\ref{interval} that $^g\!H \subset T((g_1,h_1),(g_2,h_2)).$ Now let $S_g'$ be a set of two non-adjacent vertices from $S_g$ (if $S_g$ is a clique, let $S_g'=\emptyset$). Then $(S-S_g)\cup S_g' $ is a toll set of smaller size than $S$, a contradiction. \qed
\end{proof}

\begin{theorem}\label{exact}
Let $G$ and $H$ be an arbitrary non-trivial graphs where $H$ is not a complete graph. Then $$tn(G \circ H)=\min\{|A|+2\cdot |B|+tn(H)\cdot |C|:~(A,B,C) \textrm{ is a toll-dominating triple of } G\}.$$
\end{theorem}
\begin{proof}
Let $(A,B,C)$ be a toll-dominating triple of $G$. We will prove that there exists a toll set of $G\circ H$ of size $|A|+2\cdot |B|+tn(H)\cdot |C|.$ Let $D$ be a toll set of $H$ and $h_1,h_2$ arbitrary non-adjacent vertices from $D$. We will prove that $S=(A \times \{h_1\}) \cup (B \times \{h_1,h_2\}) \cup C \times D$ is a toll set of $G \circ H.$ Let $(x,y)$ be an arbitrary vertex from $V(G \circ H).$ Suppose first that $x \in C.$ Since $D$ is a toll set of $H$, there exist $h,h'\in D$ and a tolled $h,h'$-walk $W$ in $H$ containing $y$. Then $(x,h),(x,h') \in S$ and $\{x\} \times W$ is a tolled $(x,h),(x,h')$-walk containing $(x,y)$. Suppose now that $x \notin C$. Since $(A,B,C)$ is a toll-dominating triple of $G$, there exist $u,v \in A \cup B \cup C$ such that $(x \in T_G(u,v)-\{u,v\})$ or there exists $w \in B \cup C$ such that $x \in N_G(w).$ If there exist $u,v \in A \cup B \cup C$ such that $x \in T_G(u,v)-\{u,v\}$, then Lemma~\ref{interval} implies that $(x,y) \in T_{G \circ H}((u,h_1),(v,h_1))$. Suppose now that there exists $w \in B \cup C$ such that $w$ is adjacent to $x$. Then $(w,h_1),(x,y),(w,h_2)$ is a tolled walk between two vertices from $S$ containing $(x,y).$

For the converse let $S$ be a minimum toll set of $G \circ H.$ We will prove that there exists a toll-dominating triple $(A,B,C)$ with $|A|+2|B|+tn(H)|C|=|S|.$  Let $A=\{u \in V(G):~ |^u\!H \cap S|=1\}$, $C=\{u \in V(G):~ |^u\!H \cap S|=tn(H)\}$ and $B=\{u \in V(G):~ |^u\!H \cap S|=2\}$ (if $tn(H)=2$, let $B=\emptyset$). Note that it follows from Lemma~\ref{012} that $|S|=|A|+2\cdot |B|+tn(H)\cdot |C|.$ We will prove that $(A,B,C)$ is a toll-dominating triple of $G$. Suppose that there exists $x \notin C$ such that $x \notin T_G(u,v)$ for any $u,v \in A \cup B \cup C$ different from $x$. We will prove that there exists $w \in B \cup C$ that is adjacent to $x$. Suppose first that $x \in A$ and let $(x,h)$ be an arbitrary vertex not in $S$. Since just one vertex from $^x\!H$ lies in $S$ and $S$ is a toll set of $G \circ H$, it follows from Lemma~\ref{interval}, that $(x,h) \in T_{G \circ H}((g_1,h_1),(g_2,h_2))$ for $g_1,g_2 \neq x$. If $g_1=g_2$, then it follows from Remark~\ref{sloj} that $g_1$ is adjacent to $x$ and as $|S \cap (^{g_1}\!H)| \geq 2$, $g_1 \in B \cup C.$ If $g_1 \neq g_2$, then it follows from Lemma~\ref{interval} that $x \in T_G(g_1,g_2)-\{g_1,g_2\}$, a contradiction. If $x \in B$, let $S \cap (^x\!H)=\{(x,h_1),(x,h_2)\}$. Since $x \in B$, $B$ is not empty and hence $tn(H) > 2$. Let $(x,h)$ be an arbitrary vertex not contained in the toll interval between $(x,h_1)$ and $(x,h_2)$. Since  $S$ is a toll set of $G \circ H$, it follows from Lemma~\ref{interval}, that $(x,h) \in T_{G \circ H}((g_1,h_1'),(g_2,h_2'))$ for $g_1,g_2 \neq x$. If $g_1=g_2$, then it follows from Remark~\ref{sloj} that $g_1$ is adjacent to $x$ and as $|S \cap (^{g_1}\!H)| \geq 2$, $g_1 \in B \cup C.$ If $g_1 \neq g_2$, then it follows from Lemma~\ref{interval} that $x \in T_G(g_1,g_2)-\{g_1,g_2\}$, a contradiction. Finally let $x \in V(G)-(A \cup B \cup C).$ Since $S$ is a toll set of $G \circ H$, there exist $(g_1,h_1),(g_2,h_2) \in S$ such that $(x,h) \in T_{G \circ H}((g_1,h_1),(g_2,h_2))$ for any $h \in V(H).$ As $(g_1,h_1),(g_2,h_2) \in S$, $g_1,g_2 \in A \cup B \cup C$ and hence Lemma~\ref{interval} implies that $x \in T_G(g_1,g_2)-\{g_1,g_2\}$, a final contradiction. \qed
\end{proof}

\section*{Acknowledgements}

This research was supported by the internationalisation of Slovene higher education
within the framework of the Operational Programme for Human Resources Development 2007--2013 
and by the Slovenian Research Agency project L7--5459. 

Research of T.\ Gologranc was also supported by Slovenian Research Agency under the grants N1-0043 and  P1-0297.


\end{document}